\documentclass[oneside,12pt]{amsart}
\usepackage{amsmath, amsfonts,amsthm,times,graphics}

 \makeatletter
\renewcommand*\subjclass[2][2010]{%
  \def\@subjclass{#2}%
  \@ifundefined{subjclassname@#1}{%
    \ClassWarning{\@classname}{Unknown edition (#1) of Mathematics
      Subject Classification; using '2010'.}%
  }{%
    \@xp\let\@xp\subjclassname\csname subjclassname@#1\endcsname
  }%
}

 \makeatother
\newtheorem{theorem}{Theorem}[section]
\newtheorem{lemma}[theorem]{Lemma}
\newtheorem{corollary}[theorem]{Corollary}
\newtheorem{remark}[theorem]{Remark}

\theoremstyle{definition}

 \makeatletter
\renewcommand*\subjclass[2][2010]{%
  \def\@subjclass{#2}%
  \@ifundefined{subjclassname@#1}{%
    \ClassWarning{\@classname}{Unknown edition (#1) of Mathematics
      Subject Classification; using '1991'.}%
  }{%
    \@xp\let\@xp\subjclassname\csname subjclassname@#1\endcsname
  }%
}
 \makeatother

\begin{document}

\title[On the mod $p^2$ determination of  
$\sum_{k=1}^{p-1}H_k/(k\cdot 2^k)$...]{On the mod $p^2$ determination 
of $\sum_{k=1}^{p-1}H_k/(k\cdot 2^k)$: another 
proof of  a  conjecture by Sun}

\author{Romeo Me\v strovi\' c}
\address{Department of Mathematics,
Maritime Faculty, University of Montenegro, 
Dobrota 36, 85330 Kotor, Montenegro 
romeo@ac.me}

\begin{abstract}
 For a positive integer $n$ let $H_n=\sum_{k=1}^{n}1/k$
be the $n$th  harmonic number.
Z. W. Sun conjectured  
that for any prime $p\ge 5$, 
   $$
\sum_{k=1}^{p-1}\frac{H_k}{k\cdot 2^k}
\equiv\frac{7}{24}pB_{p-3}\pmod{p^2}.
   $$
This conjecture is recently confirmed by  Z. W. Sun and L. L. Zhao.
In this note we give another proof 
of the above congruence by establishing 
congruences for all the sums of the form 
$\sum_{k=1}^{p-1}2^{\pm k}H_k^r/k^s \,(\bmod{\, p^{4-r-s}})$ with 
$(r,s)\in\{(1,1),(1,2),(2,1) \}$.    
   \end{abstract}
  \maketitle
{\renewcommand{\thefootnote}{}\footnote{2010 {\it Mathematics Subject 
Classification.} Primary 11B75; Secondary 11A07, 11B68, 05A19, 05A10.

{\it Keywords and phrases.} Harmonic numbers of order $m$, congruence,
 Bernoulli numbers.}
\setcounter{footnote}{0}}

\section{The Main Results}

Given  positive integers $n$ and $m$, the 
{\it harmonic numbers of order} $m$ are those rational numbers
$H_{n,m}$ defined as 
  $$
H_{n,m}=\sum_{k=1}^{n}\frac{1}{k^m}.
  $$
For simplicity, we will denote by
 $$
H_n:=H_{n,1}=\sum_{k=1}^n\frac{1}{k}
  $$   
the $n$th  {\it harmonic number} (in addition, we define $H_0=0$).

Recently, Z. W. Sun \cite{s}  obtained  basic congruences modulo a prime
$p\ge 5$ for several sums of terms involving harmonic numbers.
In particular,  Sun established
$\sum_{k=1}^{p-1}H_k^r\,(\bmod{\, p^{4-r}})$ for $r=1,2,3$. 
Further generalizations of these congruences have been recently obtained by 
Tauraso in \cite{t}. 

Recall that Bernoulli numbers $B_0,B_1,B_2,\ldots$ are  given by
   $$
B_0=1\quad\mathrm{and}\quad \sum_{k=0}^{n}{n+1\choose k}B_k=0\,\, 
(n=1,2,3,\ldots).
  $$

In this note we establish  six congruences  
involving harmonic numbers contained in the following result.
\begin{theorem}\label{t1.1}  Let $p>5$ be a prime. Then
  \begin{equation}\label{con2}
\sum_{k=1}^{p-1}\frac{2^kH_k}{k}\equiv -q_p(2)^2+
\frac{2}{3}pq_p(2)^3+\frac{p}{12}B_{p-3}\pmod{p^2},
   \end{equation}
 \begin{equation}\label{con3}
\sum_{k=1}^{p-1}\frac{2^kH_k}{k^2}\equiv -
\frac{1}{3}q_p(2)^3+\frac{23}{24}B_{p-3}\pmod{p},
   \end{equation}
\begin{equation}\label{con4}
\sum_{k=1}^{p-1}\frac{H_k}{k^2\cdot 2^k}\equiv \frac{5}{8}B_{p-3}\pmod{p},
   \end{equation}
\begin{equation}\label{con5}
\sum_{k=1}^{p-1}\frac{2^kH_k^2}{k}\equiv -\frac{1}{3}q_p(2)^3
+\frac{11}{24}B_{p-3}\pmod{p},
   \end{equation}
\begin{equation}\label{con5.1}
\sum_{k=1}^{p-1}\frac{H_k^2}{k\cdot 2^k}\equiv 
\frac{7}{8}B_{p-3}\pmod{p}
   \end{equation}
and
\begin{equation}\label{con6}
\sum_{k=1}^{p-1}\frac{2^kH_{k,2}}{k}\equiv -\frac{1}{3}q_p(2)^3
-\frac{25}{24}B_{p-3}\pmod{p}.
   \end{equation}
\end{theorem}  
As an application, we obtain a result obtained quite recently 
by Z. W. Sun and L. L. Zhao in \cite{sz}.
\begin{corollary}\label{c1.2}
{\rm (\cite[Theorem 1.1]{sz})}  Let $p>5$ be a prime. Then
   \begin{equation}\label{con7}
\sum_{k=1}^{p-1}\frac{H_k}{k\cdot 2^k}
\equiv\frac{7}{24}pB_{p-3}\pmod{p^2}
   \end{equation}
and 
 \begin{equation}\label{con7.1}
\sum_{k=1}^{p-1}\frac{H_{k,2}}{k\cdot 2^k}
\equiv -\frac{3}{8}B_{p-3}\pmod{p}.
   \end{equation}
\end{corollary}

 \begin{remark}
{\rm The congruence (\ref{con7}) is 
conjectured by Z. W. Sun in \cite[Conjecture 1.1]{s} 
and quite recently proved by Z. W. Sun and L. L. Zhao in 
\cite{sz}. We point out that Lemma 2.3 
from \cite{sz} presents  the  main auxiliary result in the
proof of (\ref{con7}) and its proof is based on a 
polynomial congruence recently obtained by L. L. Zhao 
and Z. W. Sun in \cite[Theorem 1.2]{zs}.  
Moreover, in this proof the authors also use 
the congruence $\sum_{k=1}^{p-1}H_k/k^2 \equiv B_{p-3} \,(\bmod{\,p})$
obtained by Sun and Tauraso in \cite[the congruence (5.4)]{st}. 
 
Notice also that the first congruence in \cite[Conjecture 1.1]{s} 
is also proved by the author of this note in 
\cite[Theorem 1.1 (3)]{m}.}
 \end{remark}

Reducing the modulus in (\ref{con2}) 
we have 

\begin{corollary}\label{c1.3} Let $p>5$ be a prime. Then
  \begin{equation}\label{con8}
\sum_{k=1}^{p-1}\frac{2^kH_k}{k}\equiv -q_p(2)^2
\pmod{p}.
   \end{equation}
\end{corollary}
This paper is organized as follows. 
In the next section, using numerous classical and 
recent combinatorial congruences, 
we prove the congruence (\ref{con2}). Applying (\ref{con2}) and some
auxiliary results, in Section 3 we establish  
the congruences (\ref{con3}) and (\ref{con4}). Section 4 is devoted 
to the proof of (\ref{con5}), (\ref{con5.1}) 
and (\ref{con6}) based on the previous
congruences and an identity for harmonic numbers.
As an application, in Section 5 
we prove Corollary~\ref{c1.2}  which contains
two  congruences  recently obtained by Z. W. Sun
and L. L. Zhao in \cite{sz}.

\section{Proof of the congruence (\ref{con2})}

\begin{lemma}\label{l2.1}
 If $p\ge 3$ is a prime, then
  \begin{equation}\label{con9}
{p-1\choose k}\equiv (-1)^k-(-1)^{k}pH_{k}+(-1)^k\frac{p^2}{2}
(H_{k}^2-H_{k,2})\pmod{p^3}
   \end{equation}
for each $k=1,2,\ldots,p-1$. In particular, we have 
  \begin{equation}\label{con10}
{p-1\choose k}\equiv (-1)^k-(-1)^{k}pH_{k}\pmod{p^2}.
   \end{equation}
 \end{lemma}

\begin{proof} 
For a fixed $1\le k\le p-1$ we have
    \begin{equation*}\begin{split} 
(-1)^k{p-1\choose k}
&=\prod_{i=1}^k\left(1-\frac{p}{i}\right)
\equiv 1-\sum_{i=1}^k\frac{p}{i}+\sum_{1\le i<j\le k}\frac{p^2}{ij}
\quad\pmod{p^3}\\
&= 1-pH_k+\frac{p^2}{2}\left(\Big(\sum_{i=1}^k\frac{1}{i}\Big)^2-
\sum_{i=1}^k\frac{1}{i^2}\right)\\
&= 1-pH_k+\frac{p^2}{2}(H_k^2-H_{k,2})\quad\qquad\pmod{p^3},
  \end{split}\end{equation*}
whence we have (\ref{con9}). Notice that reducing the modulus 
into (\ref{con9}) yields (\ref{con10}).
\end{proof}

\begin{lemma}\label{l2.2}
 If $p> 3$ is a prime, then
   \begin{equation}\label{con11}
H_{p-1}:=\sum_{k=1}^{p-1}\frac{1}{k}
\equiv -\frac{p^2}{3}B_{p-3}\pmod{p^3},
   \end{equation}
\begin{equation}\label{con12}
H_{p-1,2}:=\sum_{k=1}^{p-1}\frac{1}{k^2}
\equiv \frac{2p}{3}B_{p-3}\pmod{p^2},
   \end{equation}
\begin{equation}\label{con13}
H_{p-1,3}:=\sum_{k=1}^{p-1}\frac{1}{k^3}
\equiv 0\pmod{p^2},
   \end{equation}
 \begin{equation}\label{con14}
H_{(p-1)/2}:=\sum_{k=1}^{(p-1)/2}\frac{1}{k}\equiv 
-2q_2(p)+pq_2(p)^2-\frac{2p^2}{3}q_2(p)^3-\frac{7p^2}{12}B_{p-3}  \pmod{p^3},
   \end{equation}
   \begin{equation}\label{con15}
H_{(p-1)/2,2}:=\sum_{k=1}^{(p-1)/2}\frac{1}{k^2}\equiv \frac{7p}{3}B_{p-3}
\pmod{p^2}
   \end{equation}
and
 \begin{equation}\label{con16}
H_{(p-1)/2,3}:=\sum_{k=1}^{(p-1)/2}\frac{1}{k^3}\equiv  -2B_{p-3} \pmod{p}.
   \end{equation}
\end{lemma}
  \begin{proof}
The congruence (\ref{con11}) 
is proved in \cite{l}; see also \cite[Theorem 5.1(a)]{s1},
while  (\ref{con12}) is  a
   particular case of \cite[Corollary 5.1]{s1}
The well known congruence (\ref{con13}) 
is a particular case of \cite[Theorem 3 (b)]{b}
and  (\ref{con14}) is in fact the congruence 
(c) in \cite[Theorem 5.2]{s1}. Further, the congruences 
(\ref{con15}) and (\ref{con16}) are 
the congruences (a) with $k=2$ and  
(b) with $k=3$ in \cite[Corollary 5.2]{s1}, respectively.
  \end{proof}

We will also need the following  six 
congruences recently established 
by  Z. H. Sun \cite{s2} and Dilcher and Skula \cite{ds}.
\begin{lemma}\label{l2.3}
 Let $p>3$ be a prime. Then
  \begin{equation}\label{con17}
\sum_{k=1}^{p-1}\frac{2^k}{k}\equiv -2q_p(2)-\frac{7p^2}{12}
B_{p-3}\pmod{p^3},
   \end{equation}
     \begin{equation}\label{con18} 
\sum_{k=1}^{p-1}\frac{2^k}{k^2}\equiv -q_p(2)^2+p\left(\frac{2}{3}q_p(2)^3
+\frac{7}{6}B_{p-3}\right)\pmod{p^2},
          \end{equation}
    \begin{equation}\label{con19}
\sum_{k=1}^{p-1}\frac{1}{k\cdot 2^k}\equiv q_p(2)-\frac{p}{2}q_p(2)^2
\pmod{p^2},
   \end{equation}
    \begin{equation}\label{con20} 
\sum_{k=1}^{p-1}\frac{1}{k^2\cdot 2^k}\equiv -\frac{1}{2}q_p(2)^2
\pmod{p},
   \end{equation}
     \begin{equation}\label{con21} 
\sum_{k=1}^{p-1}\frac{2^k}{k^3}\equiv -\frac{1}{3}q_p(2)^3-
\frac{7}{24}B_{p-3}\pmod{p}
   \end{equation}
and
  \begin{equation}\label{con22} 
\sum_{k=1}^{p-1}\frac{1}{k^3\cdot 2^k}\equiv \frac{1}{6}q_p(2)^3+
\frac{7}{48}B_{p-3}\pmod{p}.
   \end{equation}
 \end{lemma}
\begin{proof}
The congruences (\ref{con17})--(\ref{con20})
are in fact the congruences (i)--(iv) in \cite[Theorem 4.1]{s2}.
By the congruence (5) in \cite[Theorem 1]{ds},
     $$
  \sum_{k=1}^{p-1}\frac{2^k}{k^3}\equiv -\frac{1}{3}q_p(2)^3-
\frac{7}{48}\sum_{k=1}^{(p-1)/2}\frac{1}{k^3}\pmod{p},
  $$
from which since by  (\ref{con16}),
    $$
\sum_{k=1}^{(p-1)/2}\frac{1}{k^3}\equiv  -2B_{p-3}\pmod{p}
   $$
we immediately obtain (\ref{con21}).
Finally, (\ref{con22}) follows immediately from 
(\ref{con21}) by applying the substitution trick $k\mapsto p-k$ 
and the fact that $2^p\equiv 2 \,(\bmod{\,p})$ by Fermat
little theorem.
    \end{proof}

\begin{lemma}\label{l2.4} Let $n$ be an arbitrary positive integer. Then
       \begin{equation}\label{con23}
\sum_{1\le k\le i\le n}\frac{2^k-1}{ki}
=\sum_{j=1}^{n}\frac{1}{j^2}{n\choose j}.
   \end{equation}          
  \end{lemma}
\begin{proof} 
Using the well known identities 
$\sum_{k=j}^{i}{k-1\choose i-1}={i\choose j}$ with $j\le i$, 
 $\frac{1}{k}{k\choose j}=\frac{1}{j}{k-1\choose j-1}$ with 
$j\le k$, and the fact that ${k\choose j}=0$ when $k<j$, we have 
  \begin{eqnarray*}
\sum_{1\le k\le i\le n}\frac{2^k-1}{ki}
&=&\sum_{1\le k\le i\le n}\frac{(1+1)^k-1}{ki}=
\sum_{1\le k\le i\le n}\frac{1}{i}\sum_{j=1}^{k}\frac{1}{k}{k\choose j}\\
&=&\sum_{1\le k\le i\le n}\frac{1}{i}\sum_{j=1}^{n}
\frac{1}{j}{k-1\choose j-1}=\sum_{j=1}^{n}\frac{1}{j}
\sum_{1\le k\le i\le n}\frac{1}{i}{k-1\choose j-1}\\
&=&\sum_{j=1}^{n}\frac{1}{j}\sum_{j\le k\le i\le n}
\frac{1}{i}{k-1\choose j-1}=\sum_{j=1}^{n}\frac{1}{j}\sum_{i=j}^n\frac{1}{i}
\sum_{k=j}^i{k-1\choose j-1}\\
&=&\sum_{j=1}^{n}\frac{1}{j}\sum_{i=j}^n\frac{1}{i}{i\choose j}
=\sum_{j=1}^{n}\frac{1}{j}\sum_{i=j}^n\frac{1}{j}{i-1\choose j-1}\\
&=&\sum_{j=1}^{n}\frac{1}{j^2}\sum_{i=j}^n{i-1\choose j-1}=
\sum_{j=1}^{n}\frac{1}{j^2}{n\choose j},
  \end{eqnarray*}
as desired.
  \end{proof}
\begin{lemma}\label{l2.5}
Let $p>3$ be a prime. Then
  \begin{equation}\label{con24}
\sum_{k=1}^{p-1}\frac{1}{k^2}{p-1\choose k}\equiv \frac{3p}{4}B_{p-3}
\pmod{p^2}.
   \end{equation}
\end{lemma}
\begin{proof}
By the congruence (\ref{con10}) of Lemma~\ref{l2.1}, we have 
  \begin{equation}\label{con25}\begin{split}
\sum_{k=1}^{p-1}\frac{1}{k^2}{p-1\choose k}
&\equiv \sum_{k=1}^{p-1}\frac{(-1)^k}{k^2}(1-pH_k)\pmod{p^2}\\
&=\sum_{k=1}^{p-1}\frac{(-1)^k}{k^2}-p\sum_{k=1}^{p-1}\frac{(-1)^kH_k}{k^2}
\pmod{p^2}.
    \end{split} \end{equation}  
Using (\ref{con12}) and (\ref{con15}) of Lemma~\ref{l2.2}  we have
  \begin{equation}\label{con26}\begin{split}
\sum_{k=1}^{p-1}\frac{(-1)^k}{k^2}
&=2\sum_{1\le j\le p-1\atop 2\mid j}\frac{1}{j^2}-
\sum_{k=1}^{p-1}\frac{1}{k^2}\\
&=\frac{1}{2}\sum_{k=1}^{(p-1)/2}\frac{1}{k^2}-\sum_{k=1}^{p-1}\frac{1}{k^2}
\equiv\frac{p}{2}B_{p-3}\pmod{p^2}.
    \end{split} \end{equation}  
Similarly, by  (\ref{con13}) and (\ref{con16}) of Lemma~\ref{l2.2}  we have
  \begin{equation}\label{con27}\begin{split}
\sum_{k=1}^{p-1}\frac{(-1)^k}{k^3}
&=2\sum_{1\le j\le p-1\atop 2\mid j}\frac{1}{j^3}-
\sum_{k=1}^{p-1}\frac{1}{k^3}\\
&=\frac{1}{4}\sum_{k=1}^{(p-1)/2}\frac{1}{k^3}-\sum_{k=1}^{p-1}\frac{1}{k^3}
\equiv -\frac{1}{2}B_{p-3}\pmod{p}.
    \end{split} \end{equation} 
Since $p\mid H_{p-1}$, it follows that  for each $k=1,2,\ldots,p-1$,
  \begin{equation}\label{con28}
H_{k}=H_{p-1}-\sum_{i=1}^{p-k-1}\frac{1}{p-i}\equiv 
\sum_{i=1}^{p-k-1}\frac{1}{i}=H_{p-k-1}\pmod{p}.
 \end{equation}
Therefore, 
\begin{equation*}\begin{split}
\sum_{k=1}^{p-1}\frac{(-1)^kH_{k}}{k^2}
&=\sum_{k=1}^{p-1}\frac{(-1)^k}{k^2}
\left(H_{k-1}+\frac{1}{k}\right)\\
&=\sum_{k=1}^{p-1}\frac{(-1)^kH_{k-1}}{k^2}
+\sum_{k=1}^{p-1}\frac{(-1)^k}{k^3}
=\sum_{k=1}^{p-1}\frac{(-1)^{p-k}H_{p-k-1}}{(p-k)^2}+
\sum_{k=1}^{p-1}\frac{(-1)^k}{k^3}\\
&\equiv -\sum_{k=1}^{p-1}\frac{(-1)^{k}H_{k}}{k^2}+
\sum_{k=1}^{p-1}\frac{(-1)^k}{k^3}\pmod{p}
  \end{split}\end{equation*} 
from which taking (\ref{con27}) we have  
  \begin{equation}\label{con29}
\sum_{k=1}^{p-1}\frac{(-1)^kH_{k}}{k^2}\equiv 
\frac{1}{2}\sum_{k=1}^{p-1}\frac{(-1)^k}{k^3}\equiv -\frac{1}{4}B_{p-3}
\pmod{p}.
  \end{equation}
Finally, substituting (\ref{con26}) and (\ref{con29}) into
(\ref{con25}) we obtain (\ref{con24}).
  \end{proof}

\begin{proof}[Proof of the congruence {\rm (\ref{con2})}] 
Observe that
the identity (\ref{con23}) of Lemma~\ref{l2.4} with $n=p-1$ may be written as 
   \begin{equation}\label{con30}
\sum_{1\le k< i\le p-1}\frac{2^k}{ki}+\sum_{k=1}^{p-1}\frac{2^k}{k^2}-
\sum_{1\le k\le i\le p-1}\frac{1}{ki}
=\sum_{k=1}^{p-1}\frac{1}{k^2}{p-1\choose k}.
    \end{equation}
Further, from (\ref{con11}) of Lemma~\ref{l2.2} we see that 
$H_{p-1}\equiv 0 \,(\bmod{\,p^2})$ (the well known 
Wolstenholme's theorem \cite{a} or \cite{gr}), and thus for each 
$k=0,1,2,\ldots p-2$,
   \begin{equation*}
\sum_{i=k+1}^{p-1}\frac{1}{i}\equiv -\sum_{i=1}^{k}\frac{1}{i}
=-H_{k}\pmod{p^2}.
     \end{equation*} 
Therefore, 
  \begin{equation}\label{con31}\begin{split}
 \sum_{1\le k< i\le p-1}\frac{2^k}{ik}
&=\sum_{k=1}^{p-1}\frac{2^k}{k}\sum_{i=k+1}^{p-1}\frac{1}{i}\equiv 
-\sum_{k=1}^{p-1}\frac{2^kH_{k}}{k}\pmod{p^2}.
  \end{split} \end{equation}  
Further, from the shuffle relation 
  \begin{equation*}
2\sum_{1\le k\le i\le p-1}\frac{1}{ki}=
\left(\sum_{k=1}^{p-1}\frac{1}{k}\right)^2+\sum_{k=1}^{p-1}\frac{1}{k^2}=
H_{p-1}^2+H_{p-1,2}
  \end{equation*}
by setting the Wolstenholme's congruence 
$H_{p-1}\equiv 0 \,(\bmod{\,p^2})$ and (\ref{con12}) of Lemma~\ref{l2.2}, we 
obtain 
\begin{equation}\label{con32}
\sum_{1\le k\le i\le p-1}\frac{1}{ki}\equiv \frac{p}{3}B_{p-3}\pmod{p^2}.
  \end{equation}
Finally, substituting (\ref{con24}) of Lemma~\ref{l2.5}, 
(\ref{con18}) of Lemma~\ref{l2.3}, (\ref{con31}) and (\ref{con32}) 
into the equality (\ref{con30}), we get 
  $$
\sum_{k=1}^{p-1}\frac{2^kH_k}{k}\equiv -q_p(2)^2+
\frac{2}{3}pq_p(2)^3+\frac{p}{12}B_{p-3}\pmod{p^2}
  $$
which is the desired congruence (\ref{con2}).
 
\end{proof}

\section{Proof of the congruences (\ref{con3}) and (\ref{con4})}

 \begin{lemma}\label{l3.1} Let $n$ be a  positive integer. Then 
  \begin{equation}\label{con33}
\sum_{k=1}^{n-1}\frac{(-2)^k}{k}{n\choose k}=\left\{
\begin{array}{ll}
-2H_{n-1}+H_{(n-1)/2}+\frac{2^n-2}{n} & if\,\, n\,\, is \,\, odd\\
-2H_n+H_{n/2}-\frac{2^n}{n} & if\,\, n\,\, is \,\, even.\\
\end{array}\right.
\end{equation}
 \end{lemma}
\begin{proof}
In the proof of Lemma 4.1 in \cite{s2} it was proved 
that of each positive odd integer $n$ holds 
 \begin{equation}\label{con34}
\sum_{k=1}^{n-1}\frac{(-1)^k}{k}{n\choose k}x^k
=\sum_{k=1}^{n-1}\frac{(1-x)^k-1}{k}-\frac{1-x^n+(x-1)^n}{n},\,\, x\in \Bbb R.
   \end{equation}   
Taking $x=2$ into (\ref{con34}),  we obtain 
  \begin{equation*}\begin{split}
\sum_{k=1}^{n-1}\frac{(-2)^k}{k}{n\choose k}
&=\sum_{k=1}^{n-1}\frac{(-1)^k-1}{k}-\frac{2-2^n}{n}\\
&=-2\sum_{1\le k\le n-1\atop k\,\,\mathrm{odd}}\frac{1}{k}
+\frac{2^n-2}{n}=-2\left(\sum_{k=1}^{n-1}\frac{1}{k}-
\sum_{1\le k\le n-1\atop k\,\,\mathrm{even}\,\, k}\frac{1}{k}\right)
+\frac{2^n-2}{n}\\
&=-2\left(H_{n-1}-\frac{1}{2}H_{(n-1)/2}\right)+\frac{2^n-2}{n}\\
&=-2H_{n-1}+H_{(n-1)/2}+\frac{2^n-2}{n}.
   \end{split}\end{equation*}
This proves the first equality in (\ref{con33}).

Now suppose that $n$ is even. 
Then by the binomial formula, for each $t>0$ and $x\in \Bbb R$, we have 
\begin{equation}\label{con35}
\frac{(1-xt)^{n}-1}{t}=\sum_{k=1}^{n}
\frac{{n\choose k}(-xt)^k}{t}\,dt=
\sum_{k=1}^{n}{n\choose k}(-x)^kt^{k-1}.
  \end{equation}
Since $\int_0^1 t^{k-1}\,dt=1/k$, setting $y=1-xt$ (cf. proof of Lemma
4.1 in \cite{s2}) (\ref{con35}) gives
  \begin{equation}\label{con36}\begin{split}
\sum_{k=1}^{n}\frac{(-x)^k}{k}{n\choose k}
&=\int_0^1 \sum_{k=1}^{n}{n\choose k}(-x)^kt^{k-1}\,dt=
\int_0^1\frac{(1-xt)^{n}-1}{t}\,dt\\
&=-\frac{1}{x}\int_{1}^{1-x}\frac{x(y^{n}-1)}{1-y}\,dy=
\int_{1}^{1-x}\sum_{k=1}^{n}y^{k-1}\,dy\\
&=\sum_{k=1}^{n}\frac{(1-x)^k-1}{k}.
     \end{split}\end{equation}
Taking $x=2$ into (\ref{con36}), we obtain 
  $$
\sum_{k=1}^{n-1}\frac{(-2)^k}{k}{n\choose k}+\frac{2^n}{n}=
-2\sum_{1\le k\le n\atop k\,\,\mathrm{odd}}\frac{1}{k}=-2H_n+H_{n/2}
 $$
which yields the second identity of (\ref{con33}).
\end{proof}

 \begin{lemma}\label{l3.2} Let $p>3$ be a  prime. Then 
  \begin{equation}\label{con37}
\sum_{k=1}^{p-1}\frac{(-2)^k}{k}{p\choose k}
\equiv pq_p(2)^2-\frac{2}{3}p^2q_p(2)^3+\frac{1}{12}p^2B_{p-3}\pmod{p^3}
   \end{equation}   
and
 \begin{equation}\begin{split}\label{con38}
\sum_{k=1}^{p-1}\frac{(-2)^k}{k}{p-1\choose k}
\equiv & -2q_p(2)+pq_p(2)^2-\frac{2}{3}p^2q_p(2)^3\\
&+\frac{1}{12}p^2B_{p-3}\pmod{p^3}.
  \end{split}\end{equation}
\end{lemma}
\begin{proof}
Setting $n=p$ in the first equality 
of (\ref{con33}) of Lemma~\ref{l3.1}  and using the congruences 
(\ref{con11}) and (\ref{con14}) from Lemma~\ref{l2.2}  reduced modulo $p^2$,
 we obtain 
  \begin{equation*}\begin{split}
\sum_{k=1}^{p-1}\frac{(-2)^k}{k}{p\choose k}
&=H_{p-1}-\frac{1}{2}H_{(p-1)/2}+2q_p(2)\\
&\equiv pq_p(2)^2-\frac{2}{3}p^2q_p(2)^3+\frac{1}{12}p^2B_{p-3}\pmod{p^3}.
   \end{split}\end{equation*}
This proves (\ref{con37}). 
Taking $n=p-1$ into the second equality of (\ref{con33}) from  Lemma~\ref{l3.1}
and 
substituting the congruences (\ref{con11}) and (\ref{con14}) from 
Lemma~\ref{l2.2} 
into this, we obtain 
  \begin{equation*}\begin{split}
\sum_{k=1}^{p-1}\frac{(-2)^k}{k}{p-1\choose k}
&=\sum_{k=1}^{p-1}\frac{(-2)^k}{k}{p-2\choose k}+\frac{2^{p-1}}{p-1}\\
&=-2H_{p-1}+H_{(p-1)/2}\\
&\equiv -2q_p(2)+ pq_p(2)^2-\frac{2p^2}{3}q_p(2)^3+\frac{1}{12}p^2B_{p-3}
\pmod{p^3}.
   \end{split}\end{equation*}
This is the congruence (\ref{con38}) and the proof is completed.
\end{proof}

\begin{proof}[Proof of the congruences
 {\rm(\ref{con3})} and  {\rm(\ref{con4})}]
First notice that
 \begin{equation}\label{con39}\begin{split}
\sum_{k=1}^{p-1}\frac{2^kH_k^2}{k}
&=\sum_{k=1}^{p-1}\frac{2^k\left(H_{k-1}+\frac{1}{k}\right)^2}{k}\\
&=\sum_{k=1}^{p-1}\frac{2^kH_{k-1}^2}{k}+2\sum_{k=1}^{p-1}
\frac{2^kH_{k-1}}{k^2}+\sum_{k=1}^{p-1}\frac{2^k}{k^3}.
  \end{split}\end{equation}
Further, using (\ref{con10}) of Lemma~\ref{l2.1} and the identity
${p-1\choose k-1}=\frac{k}{p}{p\choose k}$, we find that
 \begin{equation}\label{con40}\begin{split}
\sum_{k=1}^{p-1}\frac{2^kpH_{k-1}}{k^2}
&\equiv \sum_{k=1}^{p-1}\frac{2^k}{k^2}\left(1-(-1)^{k-1}{p-1
\choose k-1}\right)\pmod{p^2}\\
&=\sum_{k=1}^{p-1}\frac{2^k}{k^2}+\sum_{k=1}^{p-1}\frac{(-2)^k}{k^2}
{p-1\choose k-1}\\
&=\sum_{k=1}^{p-1}\frac{2^k}{k^2}+
\frac{1}{p}\sum_{k=1}^{p-1}\frac{(-2)^k}{k}{p\choose k}.
  \end{split}\end{equation}
Substituting the congruences (\ref{con18}) from Lemma~\ref{l2.3} and 
(\ref{con37}) from Lemma~\ref{l3.2} into  (\ref{con40}), we 
obtain 
 \begin{equation*}\begin{split}
\sum_{k=1}^{p-1}\frac{2^kpH_{k-1}}{k^2}\equiv \frac{5p}{4}B_{p-3}\pmod{p^2},
  \end{split}\end{equation*}
or equivalently,
 \begin{equation}\label{con41}\begin{split}
\sum_{k=1}^{p-1}\frac{2^kH_{k-1}}{k^2}\equiv \frac{5}{4}B_{p-3}\pmod{p}.
  \end{split}\end{equation}
Now we have 
   \begin{equation*}\begin{split}
\sum_{k=1}^{p-1}\frac{2^kH_k}{k^2}=
\sum_{k=1}^{p-1}\frac{2^kH_{k-1}}{k^2}+\sum_{k=1}^{p-1}\frac{2^k}{k^3},
    \end{split}\end{equation*}
whence inserting (\ref{con41}) and (\ref{con21}) from Lemma~\ref{l2.3}
we immediately obtain (\ref{con3}).

Since by (\ref{con28}) $H_{p-k-1}\equiv H_k\,(\bmod{\,p})$
for each $k=1,2,\ldots p-1$ and  $2^p\equiv 2 \,(\bmod{\,p})$, we have 
\begin{equation}\label{con42}\begin{split}
\sum_{k=1}^{p-1}\frac{2^kH_{k-1}}{k^2}
&= \sum_{k=1}^{p-1}\frac{2^{p-k}H_{p-k-1}}{(p-k)^2}\\
&\equiv \sum_{k=1}^{p-1}\frac{2^{1-k}H_k}{k^2}=
2\sum_{k=1}^{p-1}\frac{H_k}{k^2\cdot 2^k}\pmod{p}.
 \end{split}\end{equation}
Comparing (\ref{con41}) and (\ref{con42}) yields 
(\ref{con4}).
\end{proof}

\section{Proof of the congruences (\ref{con5}), (\ref{con5.1}) 
and (\ref{con6})}

 \begin{lemma}\label{l4.1} Let $n$ be a  positive integer. Then 
  \begin{equation}\label{con43}
\sum_{k=1}^{n-1}\frac{(-2)^{k-1}}{k}{n\choose k-1}=\left\{
\begin{array}{ll}
\frac{2^{n-1}(1-n)}{n+1} & if\,\, n\,\, is \,\, odd\\
\frac{(n-1)2^{n-1}+1}{n+1} &if \,\, n\,\, is \,\, even.\\
\end{array}\right.
\end{equation}
 \end{lemma}
\begin{proof}
Multiplying by $-1/2$  the identity (\ref{con33}) of Lemma~\ref{l3.1},
it becomes
     \begin{equation}\label{con44}
\sum_{k=1}^{n-1}\frac{(-2)^{k-1}}{k}{n\choose k}=\left\{
\begin{array}{ll}
H_{n-1}-\frac{1}{2}H_{(n-1)/2}-\frac{2^{n-1}-1}{n} &
{\rm if\,\,} n\,\, {\rm is \,\, odd}\\
H_n-\frac{1}{2}H_{n/2}+\frac{2^{n-1}}{n} & 
{\rm if}\,\, n\,\, {\rm is \,\, even}.\\
\end{array}\right.
\end{equation}
Now the identities ${n\choose k-1}={n+1\choose k}-{n\choose k}$
and (\ref{con44})  for any odd positive integer $n$ give
    \begin{equation}\label{con45}\begin{split}
 &\sum_{k=1}^{n-1}\frac{(-2)^{k-1}}{k}{n\choose k-1}
=\sum_{k=1}^{n-1}\frac{(-2)^{k-1}}{k}{n+1\choose k}-
\sum_{k=1}^{n-1}\frac{(-2)^{k-1}}{k}{n\choose k}\\
=&\sum_{k=1}^{n}\frac{(-2)^{k-1}}{k}{n+1\choose k}-\frac{(-2)^{n-1}(n+1)}{n}
-\sum_{k=1}^{n-1}\frac{(-2)^{k-1}}{k}{n\choose k}\\
=&\left(H_{n+1}-\frac{1}{2}H_{(n+1)/2}+\frac{2^n}{n+1}\right)-
\frac{2^{n-1}(n+1)}{n}\\
-&\left(H_{n-1}-\frac{1}{2}H_{(n-1)/2}-\frac{2^{n-1}-1}{n}\right)\\
=&(H_{n+1}-H_{n-1})-\frac{1}{2}(H_{(n+1)/2}-H_{(n-1)/2})+\frac{2^n}{n+1}
-\frac{n2^{n-1}+1}{n}\\
=&\left(\frac{1}{n}+\frac{1}{n+1}\right)-\frac{1}{n+1}+\frac{2^n}{n+1}
-\frac{n2^{n-1}+1}{n}\\
=&\frac{2^{n-1}(1-n)}{n+1}.
  \end{split}\end{equation}
Similarly, using the identities ${n\choose k-1}={n+1\choose k}-{n\choose k}$
and (\ref{con44})  for even positive integer $n$ we have
    \begin{equation}\label{con46}\begin{split}
&\sum_{k=1}^{n-1}\frac{(-2)^{k-1}}{k}{n\choose k-1}
=\sum_{k=1}^{n-1}\frac{(-2)^{k-1}}{k}{n+1\choose k}-
\sum_{k=1}^{n-1}\frac{(-2)^{k-1}}{k}{n\choose k}\\
=&\sum_{k=1}^{n}\frac{(-2)^{k-1}}{k}{n+1\choose k}-\frac{(-2)^{n-1}(n+1)}{n}
-\sum_{k=1}^{n-1}\frac{(-2)^{k-1}}{k}{n\choose k}\\
=&\left(H_{n}-\frac{1}{2}H_{n/2}-\frac{2^n-1}{n+1}\right)
+\frac{2^{n-1}(n+1)}{n}\\
&-\left(H_{n}-\frac{1}{2}H_{n/2}+\frac{2^{n-1}}{n}\right)\\
=&\frac{1-2^n}{n+1}+2^{n-1}=\frac{(n-1)2^{n-1}+1}{n+1}.
  \end{split}\end{equation}
The equalities (\ref{con45}) and (\ref{con46})  
are in fact  (\ref{con43}) and the proof is completed.
\end{proof}

  \begin{lemma}\label{l4.2} Let $n$ be an arbitrary positive integer. Then
 \begin{equation}\label{con47}
(-1)^{n}\sum_{k=1}^{n-1} (-1)^{k-1}{n\choose k}2^kH_k
=(2^n-2)H_{n-1}+H_{\left[n/2\right]}
+\frac{2^n-2}{n}
 \end{equation}
where $[x]$ denotes the integer part of $x$.
\end{lemma}
\begin{proof} 
We proceed by induction on $n$. An immediate computation shows
that (\ref{con47}) is satisfied for $n=1$ and $n=2$.
For every $n=1,2,\ldots$ put 
  $$
S_n=\sum_{k=1}^{n-1} (-1)^{k-1}{n\choose k}2^kH_k.
  $$
Then using the identity ${n+1\choose k}={n\choose k}+{n\choose k-1}$
we have 
  \begin{equation}\label{con48}\begin{split}
&S_{n+1}=\sum_{k=1}^{n}(-1)^{k-1}{n+1\choose k}2^kH_k\\
=&\sum_{k=1}^{n-1} (-1)^{k-1}{n+1\choose k}2^kH_k+(-1)^{n-1}(n+1)2^nH_n\\
= &\sum_{k=1}^{n-1} (-1)^{k}{n\choose k}2^kH_k
+\sum_{k=1}^{n-1} (-1)^{k-1}{n\choose k-1}2^kH_k+(-1)^{n-1}(n+1)2^nH_n\\
=&S_n+2\sum_{k=1}^{n-1}(-1)^{k-1}{n\choose k-1}2^{k-1}\left(H_{k-1}+
\frac{1}{k}\right)+(-1)^{n-1}(n+1)2^nH_n\\
=&S_n+2\sum_{k=1}^{n-1} (-1)^{k-1}{n\choose k-1}2^{k-1}H_{k-1}
+2\sum_{k=1}^{n-1} \frac{(-1)^{k-1}}{k}{n\choose k-1}2^{k-1}\\
&+(-1)^{n-1}(n+1)2^nH_n\\
=&S_n-2\sum_{k=1}^{n-1} (-1)^{k-1}{n\choose k}2^{k}H_{k}-
2(-1)^{n-1}n2^{n-1}H_{n-1}\\
&+2\sum_{k=1}^{n-1}\frac{(-1)^{k-1}}{k}{n\choose k-1}2^{k-1}
+(-1)^{n-1}(n+1)2^nH_n\\
=&S_n-2S_n+2\sum_{k=1}^{n-1}\frac{(-1)^{k-1}}{k}{n\choose k-1}2^{k-1}\\
&-2(-1)^{n-1}n2^{n-1}\left(H_{n}-\frac{1}{n}\right)+(-1)^{n-1}(n+1)2^nH_n\\
=&-S_n+2\sum_{k=1}^{n-1}\frac{(-1)^{k-1}}{k}{n\choose k-1}2^{k-1}+
(-1)^{n-1}2^n(H_n+1).
 \end{split}\end{equation}
Notice that both equalities (\ref{con43}) from Lemma~\ref{l4.1}   
for any positive integer $n$ can be written as
\begin{equation}\label{con49}
\sum_{k=1}^{n-1}\frac{(-2)^{k-1}}{k}{n\choose k-1}=
(-1)^n\frac{(n-1)2^{n-1}}{n+1}+(1+(-1)^n)\frac{1}{2(n+1)}.
 \end{equation}
Next, substituting (\ref{con49}) into (\ref{con48}) 
multiplied by $(-1)^{n+1}$, we find that 
    \begin{equation}\label{con50}\begin{split}
&(-1)^{n+1}S_{n+1}\\
&=(-1)^nS_n+2(-1)^{n+1}\sum_{k=1}^{n-1}\frac{(-1)^{k-1}}{k}
{n\choose k-1}2^{k-1}+2^n(H_n+1)\\
&=(-1)^nS_n-2\frac{(n-1)2^{n-1}}{n+1}-(1+(-1)^n)\frac{1}{(n+1)}+2^n(H_n+1)\\
&=(-1)^nS_n+\frac{2^{n+1}-1-(-1)^n}{n+1}+2^nH_n.
  \end{split}\end{equation}
\end{proof}
By the induction hypothesis, we have  
 \begin{equation}\label{con51}
(-1)^nS_n=(-1)^{n}\sum_{k=1}^{n-1} (-1)^{k-1}{n\choose k}2^kH_k=
(2^n-2)H_{n-1}+H_{\left[n/2\right]}+\frac{2^n-2}{n}
 \end{equation}
which substituting into (\ref{con50}) gives
  \begin{equation}\label{con52}\begin{split}
&(-1)^{n+1}S_{n+1}\\
&=(2^n-2)H_{n-1}+H_{\left[n/2\right]}+\frac{2^n-2}{n}
+\frac{2^{n+1}-1-(-1)^n}{n+1}+2^nH_n\\
&=(2^n-2)\left(H_n-\frac{1}{n}\right)+2^nH_n+H_{\left[n/2\right]}+
\frac{2^n-2}{n}+\frac{2^{n+1}-1-(-1)^n}{n+1}\\
&=(2^{n+1}-2)H_n+H_{\left[n/2\right]}+\frac{(2^{n+1}-2)+(1+(-1)^n)}{n+1}\\
&=(2^{n+1}-2)H_n+\left(H_{\left[n/2\right]}+\frac{1-(-1)^n}{n+1}\right)
+\frac{2^{n+1}-2}{n+1}\\
&=(2^{n+1}-2)H_n+H_{\left[(n+1)/2\right]}+\frac{2^{n+1}-2}{n+1}.
    \end{split}\end{equation}
This concludes the induction proof.

\begin{proof}[Proof of the congruences
 {\rm(\ref{con5})}, {\rm(\ref{con5.1})},  and  {\rm(\ref{con6})}]
Using the identities ${n\choose k}=\frac{n}{k}{n-1\choose k-1}$
$(1\le k\le n)$, $H_n=H_{n-1}+1/n$  and the congruence 
(\ref{con10}) from Lemma~\ref{l2.1}, 
the left hand side of (\ref{con47}) in Lemma~\ref{l4.2} for $n=p$ is
  \begin{equation}\label{con53}\begin{split}
&\sum_{k=1}^{p-1} (-1)^{k-1}{p\choose k}2^kH_k
=\sum_{k=1}^{p-1}\frac{2^kp}{k}(-1)^{k-1}{p-1\choose k-1}
\left(H_{k-1}+\frac{1}{k}\right)\\
&\equiv \sum_{k=1}^{p-1}\frac{2^kp}{k}(1-pH_{k-1})
\left(H_{k-1}+\frac{1}{k}\right)\pmod{p^3}\\
&=p\sum_{k=1}^{p-1}\frac{2^kH_{k-1}}{k}+
p\sum_{k=1}^{p-1}\frac{2^k}{k^2}-p^2\sum_{k=1}^{p-1}\frac{2^kH_{k-1}^2}{k}-
p^2\sum_{k=1}^{p-1}\frac{2^kH_{k-1}}{k^2}\pmod{p^3}.
   \end{split}\end{equation}
Further, note that by (\ref{con2}) of Theorem~\ref{t1.1}, 
(\ref{con18}) of Lemma~\ref{l2.3} and the identity $H_{k-1}=H_k-1/k$,
  \begin{equation}\label{con54}
\sum_{k=1}^{p-1}\frac{2^kH_{k-1}}{k}=
\sum_{k=1}^{p-1}\frac{2^kH_{k}}{k}-\sum_{k=1}^{p-1}\frac{2^k}{k^2}\equiv
-\frac{13p}{12}B_{p-3}\pmod{p^2}.
 \end{equation}
Taking (\ref{con54}), (\ref{con41}) and (\ref{con18}) of Lemma~\ref{l2.3}
 into (\ref{con53}),
we find that 
  \begin{equation}\label{con55}\begin{split}
\sum_{k=1}^{p-1} (-1)^{k-1}{p\choose k}2^kH_k
\equiv  &-pq_p(2)^2+\frac{2}{3}p^2q_p(2)^3 -\frac{7p^2}{6}B_{p-3}\\
&-p^2\sum_{k=1}^{p-1}\frac{2^kH_{k-1}^2}{k}\pmod{p^3}.
   \end{split}\end{equation}
On the other hand, by (\ref{con47}) of Lemma~\ref{l4.2} with $n=p$,
  \begin{equation}\label{con56}
\sum_{k=1}^{p-1} (-1)^{k-1}{p\choose k}2^kH_k
=(2-2^p)H_{p-1}-H_{(p-1)/2}-2q_p(2).
 \end{equation}
   Furthermore, since by Wolstenholme's theorem and Fermat little theorem, 
$p^3\mid H_{p-1}(2-2^p)$, taking this and  (\ref{con14}) of Lemma~\ref{l2.2} 
into (\ref{con56}) we get 
  \begin{equation}\label{con57}
\sum_{k=1}^{p-1} (-1)^{k-1}{p\choose k}2^kH_k\equiv 
-pq_p(2)^2+\frac{2}{3}p^2q_p(2)^3 +\frac{7p^2}{12}B_{p-3}
\pmod{p^3}.
   \end{equation}
Now substituting (\ref{con57}) into (\ref{con55}), we obtain 
   $$ 
p^2\sum_{k=1}^{p-1}\frac{2^kH_{k-1}^2}{k}\equiv 
-\frac{7p^2}{4}B_{p-3}\pmod{p^3},
   $$
or equivalently,
   \begin{equation}\label{con58} 
\sum_{k=1}^{p-1}\frac{2^kH_{k-1}^2}{k}\equiv 
-\frac{7}{4}B_{p-3}\pmod{p}.
    \end{equation}
Finally, applying (\ref{con58}), (\ref{con41}) and (\ref{con21})
of Lemma~\ref{l2.3}, we have
    \begin{equation*}\begin{split}
\sum_{k=1}^{p-1}\frac{2^kH_k^2}{k}
&=\sum_{k=1}^{p-1}\frac{2^k\left(H_{k-1}+\frac{1}{k}\right)^2}{k}\\
&= \sum_{k=1}^{p-1}\frac{2^kH_{k-1}^2}{k}+
2\sum_{k=1}^{p-1}\frac{2^kH_{k-1}}{k^2}+\sum_{k=1}^{p-1}\frac{2^k}{k^3}\\
&\equiv -\frac{1}{3}q_p(2)^3+\frac{11}{24}B_{p-3}\pmod{p}.   
  \end{split}\end{equation*}
This is in fact the congruence (\ref{con5}).

In order to prove the congruence (\ref{con5.1}), 
notice that by (\ref{con28}) $H_{p-k}\equiv H_{k-1}\,(\bmod{\,p})$
for each $k=1,2,\ldots , p-1$. Hence,  using this,   the congruence 
(\ref{con58}), Fermat little theorem, and applying  
(\ref{con5}), (\ref{con3}) and (\ref{con21}),
we find that
   \begin{equation*}\begin{split}
\sum_{k=1}^{p-1}\frac{H_k^2}{k\cdot 2^k}
&=\sum_{k=1}^{p-1}\frac{H_{p-k}^2}{(p-k)\cdot 2^{p-k}}
\equiv\sum_{k=1}^{p-1}\frac{H_{k-1}^2}{(-k)\cdot 2^{1-k}}
\pmod{p}\\
&=-\frac{1}{2}\sum_{k=1}^{p-1}\frac{2^kH_{k-1}^2}{k}
=-\frac{1}{2}\sum_{k=1}^{p-1}\frac{2^k\left(H_k-\frac{1}{k}\right)^2}{k}\\
&=-\frac{1}{2}\sum_{k=1}^{p-1}\frac{2^kH_k^2}{k}+
\sum_{k=1}^{p-1}\frac{2^kH_k}{k^2}-
\frac{1}{2}\sum_{k=1}^{p-1}\frac{2^k}{k^3}\equiv\frac{7}{8}B_{p-3}\pmod{p}.
  \end{split}\end{equation*}
This is in fact (\ref{con5.1}).

For establishing the congruence (\ref{con6}), 
observe that by (\ref{con9}) of Lemma~\ref{l2.1},
  $$
p^2H_{k,2}\equiv 2-2pH_k+p^2H_k^2-2(-1)^k{p-1\choose k}\pmod{p^3},
  $$
whence we have
 \begin{equation}\label{con59}\begin{split}
p^2\sum_{k=1}^{p-1}\frac{2^kH_{k,2}}{k}
\equiv & 2\sum_{k=1}^{p-1}\frac{2^k}{k}
-2p\sum_{k=1}^{p-1}\frac{2^kH_k}{k}+p^2\sum_{k=1}^{p-1}\frac{2^kH_k^2}{k}\\
& -2\sum_{k=1}^{p-1}\frac{(-2)^k}{k}{p-1\choose k}\pmod{p^3}.
 \end{split}\end{equation}
Finally, substituting the congruences (\ref{con17}) of Lemma~\ref{l2.3}, 
(\ref{con2}), (\ref{con5}) of Theorem~\ref{t1.1} and 
(\ref{con38}) of Lemma~\ref{l3.2} into (\ref{con59}), we obtain
  $$
p^2\sum_{k=1}^{p-1}\frac{2^kH_{k,2}}{k}\equiv -\frac{1}{3}p^2q_p(2)^3
-\frac{25}{24}p^2B_{p-3}\pmod{p^3}
   $$ 
from which we get 
 $$
\sum_{k=1}^{p-1}\frac{2^kH_{k,2}}{k}\equiv -\frac{1}{3}q_p(2)^3
-\frac{25}{24}B_{p-3}\pmod{p}.
 $$
This is the congruence (\ref{con6}), and the proof is completed.
\end{proof}

\section{Proof of Corollary~\ref{c1.2}}

\begin{lemma}\label{l5.1} If $p> 3$ is a prime, then
 \begin{equation}\label{con60}
 \sum_{1\le k\le i\le p-1}\frac{2^k}{ik^2}\equiv -\frac{5}{4}B_{p-3}\pmod{p},
   \end{equation}
  \begin{equation}\label{con61}
   \sum_{1\le k\le i\le p-1}\frac{2^k}{i^2k}\equiv \frac{3}{4}B_{p-3}\pmod{p},
     \end{equation}
   \begin{equation}\label{con62}
\sum_{1\le k\le i\le p-1}\frac{2^k}{ik}\equiv\frac{13}{12}pB_{p-3}\pmod{p^2}.
   \end{equation}
\end{lemma}
\begin{proof} Since 
$H_{p-1}\equiv 0 \,(\bmod{\,p^2})$ (the well known 
Wolstenholme's theorem), and thus for each $k=1,2,\ldots p-1$,
   \begin{equation}\label{con63}
\sum_{i=k}^{p-1}\frac{1}{i}\equiv -\sum_{i=1}^{k-1}\frac{1}{i}
=-H_{k-1}\pmod{p^2}.
     \end{equation}
Applying (\ref{con63}), (\ref{con10}) of Lemma~\ref{l2.1} and taking the 
identity ${p-1\choose k-1}=
\frac{k}{p}{p\choose k}$,  we find that
 \begin{equation}\label{con64}\begin{split}
 p\sum_{1\le k\le i\le p-1}\frac{2^k}{ik^2}
&=\sum_{k=1}^{p-1}\frac{2^k}{k^2}\sum_{i=k}^{p-1}\frac{p}{i}\equiv 
-\sum_{k=1}^{p-1}\frac{2^k}{k^2}pH_{k-1}\pmod{p^2}\\
&\equiv 
-\sum_{k=1}^{p-1}\frac{2^k}{k^2}\left(1-(-1)^{k-1}{p-1\choose k-1}\right)\\
&=-\sum_{k=1}^{p-1}\frac{2^k}{k^2}-\sum_{k=1}^{p-1}\frac{(-2)^k}{k^2}
{p-1\choose k-1}\\
&=-\frac{1}{p}\sum_{k=1}^{p-1}\frac{2^k}{k^2}-\sum_{k=1}^{p-1}\frac{(-2)^k}{k}
{p\choose k}.
 \end{split}\end{equation}
Finally, taking the congruence 
(\ref{con18})  of Lemma~\ref{l2.3} and (\ref{con37})
of Lemma~\ref{l3.2} into the right hand side of (\ref{con64}), 
we immediately obtain (\ref{con60}).

Further, from (\ref{con12}) of Lemma~\ref{l2.2} we see that 
$H_{p-1,2}\equiv 0 \,(\bmod{\,p})$ and therefore, for each 
$k=1,2,\ldots p-1$, 
   $$
\sum_{i=k}^{p-1}\frac{1}{i^2}\equiv -\sum_{i=1}^{k-1}\frac{1}{i^2}
=-H_{k-1,2}\pmod{p}.
  $$
Applying this  we obtain
 \begin{equation}\label{con65}\begin{split}
 \sum_{1\le k\le i\le p-1}\frac{2^k}{i^2k}
&=\sum_{k=1}^{p-1}\frac{2^k}{k}\sum_{i=k}^{p-1}\frac{1}{i^2}\equiv 
-\sum_{k=1}^{p-1}\frac{2^kH_{k-1,2}}{k}\pmod{p}.
    \end{split}\end{equation}
Further, taking $H_{k-1,2}=H_{k,2}-1/k^2$,
by  (\ref{con5}) of Theorem~\ref{t1.1} and (\ref{con21}) of Lemma~\ref{l2.3}, we get
   \begin{equation}\label{con66}\begin{split}    
\sum_{k=1}^{p-1}\frac{2^kH_{k-1,2}}{k}=
\sum_{k=1}^{p-1}\frac{2^kH_{k,2}}{k}-
\sum_{k=1}^{p-1}\frac{2^k}{k^3}\equiv -\frac{3}{4}B_{p-3}\pmod{p}.
    \end{split}\end{equation}
Inserting (\ref{con66}) into (\ref{con65}) we obtain (\ref{con61}). 

Finally, by  (\ref{con63}) we have 
  \begin{equation*}\begin{split}
 \sum_{1\le k\le i\le p-1}\frac{2^k}{ik}
&=\sum_{k=1}^{p-1}\frac{2^k}{k}\sum_{i=k}^{p-1}\frac{1}{i}\equiv 
-\sum_{k=1}^{p-1}\frac{2^k}{k}H_{k-1}\pmod{p^2}\\
&=-\sum_{k=1}^{p-1}\frac{2^kH_k}{k}+\sum_{k=1}^{p-1}\frac{2^k}{k^2}\pmod{p^2}
  \end{split} \end{equation*}  
whence substituting the congruences 
(\ref{con3}) of Theorem~\ref{t1.1}  and (\ref{con18}) of Lemma~\ref{l2.3},
we obtain (\ref{con62}).
    \end{proof}

We are now ready to prove the congruence 
(\ref{con7}) from Corollary~\ref{c1.2} conjectured by Z. W. Sun.

\begin{proof}[Proof of the congruence  {\rm(\ref{con7})}]
Since $1/(p-k)\equiv -(p+k)/k^2\,(\bmod{\,p^2})$,
we find that
 \begin{equation}\label{con67}\begin{split}
2^p\sum_{k=1}^{p-1}\frac{H_k}{k\cdot 2^k}
=&2^p \sum_{k=1}^{p-1}\frac{1}{k\cdot 2^k}
\sum_{i=1}^{k}\frac{1}{i}=2^p\sum_{1\le i\le k\le p-1}
\frac{1}{ik\cdot 2^k}\\
=&2^p\sum_{1\le p-i\le p-k\le p-1}\frac{1}{(p-i)(p-k)2^{p-k}}\\
\equiv & \sum_{1\le k\le i\le p-1}\frac{(p+i)(p+k)2^k}{i^2k^2}\pmod{p^2}\\
\equiv& \sum_{1\le k\le i\le p-1}\frac{(pi+pk+ik)2^k}{i^2k^2}\pmod{p^2}\\
=&p\left(\sum_{1\le k\le i\le p-1}\frac{2^k}{ik^2}+
\sum_{1\le k\le i\le p-1}\frac{2^k}{i^2k}\right)\\
&+\sum_{1\le k\le i\le p-1}\frac{2^k}{ik}\pmod{p^2}.
   \end{split}\end{equation}
The substitution of congruences 
(\ref{con60})--(\ref{con62}) of Lemma~\ref{l5.1} into 
(\ref{con67}) immediately yields 
 \begin{equation*}
2^p\sum_{k=1}^{p-1}\frac{H_k}{k\cdot 2^k}\equiv 
\frac{7}{12}pB_{p-3}\pmod{p^2}, 
  \end{equation*}
whence because of  Fermat little theorem  $2^{-p}\equiv 2^{-1}\,(\bmod{\,p})$, 
we obtain 
 \begin{equation*}
\sum_{k=1}^{p-1}\frac{H_k}{k\cdot 2^k}\equiv   
2^{-p}\frac{7}{12}pB_{p-3}\equiv \frac{7}{24}pB_{p-3} \pmod{p^2}, 
  \end{equation*}
as desired.
\end{proof}
\begin{proof}[Proof of the congruence {\rm(\ref{con7.1})}]
By (\ref{con9}) of Lemma~\ref{l2.1},
  $$
p^2H_{k,2}\equiv 2-2pH_k+p^2H_k^2-2(-1)^k{p-1\choose k}\pmod{p^3},
  $$
whence we have
 \begin{equation}\label{con59.1}\begin{split}
p^2\sum_{k=1}^{p-1}\frac{H_{k,2}}{k\cdot 2^k }
\equiv & 2\sum_{k=1}^{p-1}\frac{1}{k\cdot 2^k}
-2p\sum_{k=1}^{p-1}\frac{H_k}{k\cdot 2^k}+p^2\sum_{k=1}^{p-1}
\frac{H_k^2}{k\cdot 2^k}\\
&-2\sum_{k=1}^{p-1}\frac{(-1)^k}{k\cdot 2^k}{p-1\choose k}\pmod{p^3}.
 \end{split}\end{equation}
Taking $n=p-1$ and $x=2$ into (\ref{con36}) from the proof of 
Lemma~\ref{l3.1}, we obtain 
  \begin{equation}\label{con59.2}
\sum_{k=1}^{p-1}\frac{(-1)^k}{k\cdot 2^k}{p-1\choose k}=
\sum_{k=1}^{p-1}\frac{1}{k\cdot 2^k}-H_{p-1}.
  \end{equation}
Substituting (\ref{con59.2}) into (\ref{con59.1}) yields
  \begin{equation}\label{con59.3}
p^2\sum_{k=1}^{p-1}\frac{H_{k,2}}{k\cdot 2^k }
\equiv  -2p\sum_{k=1}^{p-1}\frac{H_k}{k\cdot 2^k}+p^2\sum_{k=1}^{p-1}
\frac{H_k^2}{k\cdot 2^k}+2H_{p-1}\pmod{p^3}.
 \end{equation}
Finally, substituting the congruences 
(\ref{con7}) of Corollary~\ref{c1.2}, (\ref{con5.1}) of 
Theorem~\ref{t1.1} and (\ref{con11}) of 
Lemma~\ref{l2.2} into (\ref{con59.3}), we obtain
  $$
p^2\sum_{k=1}^{p-1}\frac{H_{k,2}}{k\cdot 2^k}\equiv 
-\frac{3}{8}p^2B_{p-3}\pmod{p^3}
   $$ 
whence it follows that 
 $$
\sum_{k=1}^{p-1}\frac{H_{k,2}}{k\cdot 2^k}\equiv -\frac{3}{8}B_{p-3}\pmod{p}.
 $$
This is the congruence (\ref{con7.1}), and the proof is completed.
\end{proof}

\end{document}